\documentclass[12pt,a4paper]{article}

\usepackage{amssymb, latexsym, amsthm}
\usepackage{graphicx} 
\usepackage{url}      
\newcommand{\doi}[1]{\url{http://dx.doi.org/#1}}
\usepackage{amsmath}  
\allowdisplaybreaks

\IfFileExists{ajr.sty}{\usepackage{ajr}}{}

\newcommand{\rat}[2]{{\textstyle\frac{#1}{#2}}}

\newcommand{\p}{\partial}

\renewcommand{\phi}{\varphi}

\newcommand{\R}{{\mathbb R}}

\newcommand{\spde}{\textsc{spde}}
\newcommand{\sde}{\textsc{sde}}
\newcommand{\Ord}[1]{\ensuremath{\mathcal O\big(#1\big)}}

\newtheorem{theorem}{Theorem}
\newtheorem{lemma}[theorem]{Lemma}

\newtheorem{definition}[theorem]{Definition}
\theoremstyle{definition}
\newtheorem{remark}[theorem]{Remark}

\title{Macroscopic  discrete modelling of stochastic reaction-diffusion equations on a periodic domain}

\author{
W. Wang\thanks{School of Mathematics, University of Adelaide,
South Australia, \textsc{Australia}, \protect\url{mailto:
w.wang@adelaide.edu.au}; and Department of
Mathematics, Nanjing University, Nanjing, \textsc{China},
\protect\url{mailto:wangweinju@yahoo.com.cn}}
\and
A.~J. Roberts\thanks{School of Mathematics, University of Adelaide,
South Australia, \textsc{Australia}, \protect\url{mailto:
anthony.roberts@adelaide.edu.au}}
}

\date{\today}

\begin{document}

\maketitle

\begin{abstract}
Dynamical systems theory provides powerful methods to extract
effective macroscopic dynamics from complex systems with slow modes
and fast modes. Here we derive and theoretically support a
macroscopic, spatially discrete, model for a class of stochastic
reaction-diffusion partial differential equations with cubic nonlinearity. Dividing space
into overlapping finite elements, a special coupling condition between
neighbouring elements preserves the self-adjoint dynamics and
controls interelement interactions.  When the interelement coupling
parameter is small, an averaging method and an asymptotic expansion
of the slow modes show that the macroscopic discrete model will be
a family of coupled stochastic ordinary differential equations which
describe the evolution of the grid values. This modelling shows the
importance of subgrid scale interaction between noise and spatial diffusion
and provides a new rigourous approach to constructing semi-discrete
approximations to stochastic reaction-diffusion partial differential equations.
\end{abstract}


\section{Introduction}\label{sec:intro}
In modelling complex systems we often desire to model the effective macroscopic dynamics. Some cases can be extracted from the full, or microscopic, description by methods such as averaging, invariant manifold reduction and homogenization~\cite[e.g.]{GKS04, WD07-1}.

Stochastic partial differential equations (\textsc{spde}s) are widely studied in modeling, analyzing, simulating and predicting complex phenomena in many fields of nonlinear science~\cite[e.g.]{E00, Imkeller, Simulation, WaymireDuan}. Recently, macroscopic reduction for dissipative \textsc{spde}s, with two widely separated timescales, has been studied by the dynamical systems theory of stochastic invariant  manifolds~\cite[e.g.]{Rob03, Rob06, WD07} and by averaging methods~\cite[e.g.]{WR08, WR08-1}.  Moreover, invariant manifold theory also applies to generate a macroscopic discrete model of deterministic and stochastic dissipative \textsc{pde}s~\cite[e.g.]{Rob02, Rob07, Rob09}, the so-called holistic finite differences. Roberts~\cite{Rob08} recently extended the approach to ensure macroscopic discrete models preserve important self-adjoint properties of the fine scale dynamics for deterministic systems.

Here we address the macroscopic discrete modelling of dissipative \textsc{spde}s and develop a novel rigorous approach. We consider reaction-diffusion  in a one dimensional spatial domain driven by a noise which is white in time and with spatial structure. Let the non-dimensional spatial interval $I=[0, L]$ with length $L>0$, and let $L^2(I)$ be the Lebesgue space of square integrable real valued functions on $I$. Consider the following non-dimensional stochastic reaction-diffusion equation for a stochastic field~$u(x,t)$, of period~$L$ in space~$x$,
\begin{eqnarray}
\p_tu&=&\p_{xx} u+\alpha(u-u^3)+\sigma\,\p_tW
\quad \text{on } I, \label{e:sRDe}\\
u(0, t)&=&u(L, t)
\quad\text{and}\quad u_x(0, t)=u_x(L, t),
\label{e:sRDe-bd}
\end{eqnarray}
where $W(x, t)$ is an $L^2(I)$ valued $Q$-Wiener process defined on
a complete probability space $(\Omega, \mathcal{F}, \mathbb{P})$
which is detailed in the next section.

The spatial domain~$I$ is divided into $M$~elements and consequently a set of $M$~fields are defined, one on each of these elements. In order to preserve the self-adjoint property of the linear operator defined on the elements we impose special interelement coupling conditions with a strength parametrised by~$\gamma$.
The system dynamics are expanded in the interelement coupling~$\gamma$ so that based upon the case of weak coupling, that is, small $\gamma>0$\,,  an asymptotic approximation and an averaging method derive a family of coupled stochastic differential equations (\sde{}s) which describe the evolution of grid values; that is, the amplitude of the system on each element.  These \sde{}s are a discrete stochastic model of the continuous space stochastic system~\eqref{e:sRDe}--\eqref{e:sRDe-bd}.  Our discrete \sde\ model highlights the macroscopic influence of `subgrid' interactions between noise and spatial diffusion in the \spde.

The simplest conventional finite difference approximation of the \spde~\eqref{e:sRDe} on a regular grid in~$x$, say $X_j=jh$ for some constant grid spacing~$h$, is
\begin{equation}\label{e:diff-appr}
dU_j=\frac{1}{h^2}(U_{j+1}-2U_j+U_{j-1})\,dt+\alpha(
U_j-U_j^3)dt+\sigma\, dW_j
\end{equation}
where $U_j(t)$ is the grid value of the field~$u(x,t)$ at the grid
points~$X_j$, and similarly $W_j(t)=W(X_j,t)$.  However, our
analysis herein recommends that a more accurate closure incorporates
stochastic influences from the neighbour grids as
in the Ito system of \sde{}s
\begin{eqnarray}
dU_j&=&\frac{1}{h^2}(U_{j+1}-2U_j+U_{j-1})\,dt+(\hat{\alpha}
U_j-\alpha U_j^3)dt+\sigma\, dW_j\nonumber\\&&{}+ 3\sqrt{2}
U_jd\check{W}_j+\frac{\sigma}{4}\Big(dW_{j+1}-2dW_j+dW_{j-1}\Big)\,.\label{e:Uj}
\end{eqnarray}
The second terms in the last line of the above \sde\ system reflect
interaction between noise and spatial diffusion. Terms
in~$\check{W}_j$ and~$\hat{\alpha}$ are due to the microscopic,
subgrid scale, stochastic interactions discussed in
Sections~\ref{sec:macr-model}.


In order to generate a macroscopic discrete model we divide the domain into finite overlapping elements and choose special coupling boundary conditions, Section~\ref{sec:coulping BDC}. Such interelement coupling rules were first introduced by Roberts~\cite{Rob08} to construct spatially discrete models of deterministic dynamics. One important property of this interelement coupling is the preservation of self-adjoint symmetry in the underlying spatial dynamics. Moreover, the strength of the coupling is parametrised by~$\gamma$, $0\leq\gamma\leq 1$\,: when the coupling parameter~$\gamma$ is small, the coupling is weak and the system separates into `uninterestingly' decaying fast parts and the relevant slow parts, Section~\ref{sec:macr-model}. Then an averaging method~\cite{WR08}  derives a reduced model which describes the evolution of the overall amplitude of~\eqref{e:sRDe}--\eqref{e:sRDe-bd} on the whole domain. Further, an analysis on eigenfunctions obtains a reduced model describing the evolution of the local amplitude on each element. This model is the macroscopic discrete approximation to~\eqref{e:sRDe}--\eqref{e:sRDe-bd} expressed in~\eqref{e:Uj}.

In this approach one difficulty is to construct from the original spatio-temporal noise~$W(x,t)$ a Wiener process~$W^\gamma_j(x,t)$ on each element. A natural method is expanding $W(x,t)$ by the eigenfunctions of the linear operator~$\mathcal{L}_\gamma$ on each element, Section~\ref{sec:coulping BDC}, which has analogues in the method of finite elements~\cite{Thom97}. This construction shows that for small~$h$ and small coupling~$\gamma$, the stochastic force on the first mode of $\mathcal{L}_\gamma$ approximates the grid values~$W(X_j,t)$. Moreover, analysis of the eigensystem of $\mathcal{L}_\gamma$ for small~$\gamma$ shows the slow parts of the system dominate: the fast parts converge to quasi-equlibrium with rate~$1/h^2$. Then by this and the construction of~$W^\gamma(x,t)$, the macroscopic discrete reduced  model is proved to be consistent to the stochastic reaction-diffusion equations~\eqref{e:sRDe}--\eqref{e:sRDe-bd} as the element size $h\rightarrow 0$\,.

Another difficulty is that the linear operator~$\mathcal{L}_\gamma$\,, defined in~\eqref{e:L-gamma}, varies with the coupling parameter~$\gamma$. So the continuity of the linear oprator~$\mathcal{L}_\gamma$ in~$\gamma$ is needed. Section~\ref{sec:asy} argues that the graph convergence of~$\mathcal{L}_\gamma$ as $\gamma\rightarrow 0$ ensures the continuity of eigenfunctions and eigenvalues in~$\gamma$, then an asymptotic expansion of the first eigenfunction in~$\gamma$ shows that the grid value is the amplitude of the system on the element. Then by averaging and the expansion of the first eigenfunction of~$\mathcal{L}_\gamma$, the macroscopic reduced model is in the first eigenspace of~$\mathcal{L_\gamma}$, which is varying with respect to~$\gamma$. So last we project the reduced model to the first eigenspace of~$\mathcal{L}_0$, the basic mode, to derive the macroscopic model. In this approach one interesting phenomenon is that the effect of noise in the subgrid scale fast modes is transmitted into the macroscopic slow modes by the projection. Numerical simulations confirm such transmittal~\cite{Rob07}.


\section{Overlapping finite elements and coupling boundary conditions}\label{sec:coulping BDC}
This section divides the spatial domain~$I$ into $M$~overlapping elements with grid spacing~$h$. Let the $j$th~element
\begin{equation*}
I_j=\left[X_j-h, X_j+h\right]
\end{equation*}
with grid points $X_j=jh$ on each element~$I_j$, $j=1, 2, \ldots,
M$\,. Here for the periodic boundary condition we use the notation
$X_{j\pm M}=X_j$\,. Let~$u_j(x)$ denote the field on the
element~$I_j$, $j=1, 2, \ldots, M$\,. Denote by $f(u)=-u^3$\,. Then, modified from the \spde~\eqref{e:sRDe},
consider the following system of \spde{}s defined on elements~$I_j$, $j=1, 2,
\ldots, M$\,,
\begin{equation}\label{e:uj}
\p_tu^\gamma_j(x,t)=\p_{xx}
u^\gamma_j(x,t)+\alpha\gamma^2u^\gamma_j(x,t)+\alpha
f(u^\gamma_j(x,t))+\sigma\p_tW^\gamma_j(x,t) \quad \text{on } I_j,
\end{equation}
with the following interelement coupling conditions on the fields parametrised by~$\gamma$, and with $\gamma'+\gamma=1$\,,
\begin{equation}\label{cbd1}
u^\gamma_j(X_{j\pm 1}, t)=\gamma'u^\gamma_j(X_j, t)+\gamma
u^\gamma_{j\pm 1}(X_{j\pm 1},t), \quad j=1, 2, \ldots ,M\,,
\end{equation}
and coupling of the first spatial derivative, denoted by subscript~$x$,
\begin{eqnarray}
&&u^\gamma_{j,x}(X_j^-,t)-u^\gamma_{j,x}(X_j^+,t)+\gamma
u^\gamma_{j-1,
x}(X_j,t)-\gamma u^\gamma_{j+1,x}(X_j,t)\nonumber\\
&&\quad-\gamma'u^\gamma_{j,x}(X_{j-1},t)+\gamma'u^\gamma_{j,x}(X_{j+1},
t)=0\,, \quad j=1, 2,\ldots, M\,,\label{cbd2}
\end{eqnarray}
with, to account for $L$-periodicity of
solutions,
\begin{align}&\label{e:notation1}
u^\gamma_{j\pm M}(x\pm L,t)=u^\gamma_j(x, t)
\\&\label{e:notation2}
u^\gamma_{j\pm M, x}(x\pm L,t)=u^\gamma_{j, x}(x,t)\,.
\end{align}
The coupling parameter~$\gamma$ controls
the flow of information between the two adjacent elements: when the
coupling $\gamma=0$\,, adjacent elements are decoupled; when
$\gamma=1$\,, the system is full coupled
and~\eqref{e:uj}--\eqref{cbd2} is equivalent to the dynamics of the
physical stochastic reaction-diffusion
equation~\eqref{e:sRDe}--\eqref{e:sRDe-bd}, see
Section~\ref{sec:estimate}.

The noise fields~$W_j^\gamma(x,t)$ defined on each element~$I_j$, $j=1,\ldots, M$\,, are infinite dimensional Wiener process which are detailed later from~$W(x,t)$.

Related but different interelement coupling boundary conditions empowered an earlier exploration of the non-self-adjoint interaction between noise, nonlinear advection and spatial diffusion in discretely modelling the stochastic Burgers' equation~\cite{Rob07}.

For our purposes, first we introduce a mathematical framework for system~\eqref{e:uj}--\eqref{cbd2}. Let $H_j=L^2(I_j)$  be the set of all square integrable function on~$I_j$ and $V_1=H^1(I_j)$. Denote by $\mathcal{H}=\Pi_{j=1}^MH_j$\,, $\mathcal{V}=\Pi_{j=1}^MV_j$ and
\begin{equation*}
\mathcal{H}^\alpha=\Pi_{j=1}^MH_j^\alpha\,,\quad \alpha>0\,.
\end{equation*}
Here $H_j^\alpha$ denotes the usual Sobolev space~$W^{2,\alpha}(I_j)$~\cite{Sob}. Define the inner product~$\langle\cdot, \cdot \rangle$ on~$\mathcal{H}$ as the sum of the element integrals
\begin{eqnarray*}
\langle u, v
\rangle=\sum_{j=1}^M\left[\int_{X_{j-1}}^{X_j^-} u_j(x)v_j(x)\,dx+\int_{X_j^+}^{X_{j+1}} u_j(x)v_j(x)\,dx\right]
\end{eqnarray*}
for any $u,v\in \mathcal{H}$ with $u=(u_j)$ and $v=(v_j)$\,. And denote by~$\|\cdot\|_0$ the product $L^2$-norm on space~$\mathcal{H}$. For any $\alpha\in\mathbb{Z}^+$ denote the semi-norm on~$\mathcal{H}^\alpha$ as
\begin{equation*}
\|u\|^2_\alpha=\sum_{j=1}^M\|\p^\alpha_xu_j\|_0^2\,, \quad
u=(u_j)\in\mathcal{H}^\alpha\,.
\end{equation*}

For the system~\eqref{e:uj}--\eqref{cbd2} we introduce the family of functional spaces
\begin{eqnarray*}
\mathcal{H}_\gamma = \big\{(u_j)\in \mathcal{H}
&:& u_j(X_{j\pm 1},
t)=\gamma'u_j(X_j, t)+\gamma u_{j\pm 1}(X_{j\pm
1},t),
\\&&{} j=1, 2, \ldots ,M \,, \quad
\text{with $L$-periodicty }\eqref{e:notation1}\big\}
\end{eqnarray*}
and the subspaces
\begin{eqnarray*}
\mathcal{V}_\gamma&=&\big\{(u_j)\in \mathcal{V}\cap
\mathcal{H}_\gamma :
u^\gamma_{j,x}(X_j^-,t)-u^\gamma_{j,x}(X_j^+,t)+\gamma
u^\gamma_{j-1, x}(X_j,t)
\\&&\quad{}
-\gamma u^\gamma_{j+1,x}(X_j,t)
-\gamma'u^\gamma_{j,x}(X_{j-1},t)+\gamma'u^\gamma_{j,x}(X_{j+1},
t)=0\,,
\\&&\quad{} j=1, 2,\ldots, M\,, \quad \text{with
$L$-periodicty }\eqref{e:notation2}\big\}.
\end{eqnarray*}
Then define the second order differential operator
$\mathcal{L}_\gamma:D(\mathcal{L}_\gamma)\subset\mathcal{V}_\gamma\rightarrow
\mathcal{H}$ by
\begin{equation}\label{e:L-gamma}
\mathcal{L}_\gamma u=\left(\frac{\p^2 u_j}{\p x^2}\right) \quad
\text{for all } u=(u_j)\in D(\mathcal{L}_\gamma).
\end{equation}
By a basic calculation~\cite{Rob08}, $-\mathcal{L}_\gamma$ is a self-adjoint second order operator. A direct calculation yields that for any $u=(u_j)\in D(\mathcal{L}_\gamma)$
\begin{equation}\label{e:H1}
\langle -\mathcal{L}_\gamma u, u \rangle=
\sum_{j=1}^M\|u_j\|_1^2\geq 0
\end{equation}
which establishes the positivity of operator $-\mathcal{L}_\gamma$\,. Then there are coupling dependent eigenfunctions $\{(e_{j,k}^\gamma(x))\}_{k=0}^\infty$ which form a standard orthonormal system in space~$\mathcal{H}_\gamma$ and a sequence of real numbers $0<\lambda_0(\gamma)\leq \lambda_1(\gamma)\leq \cdots$ such that
\begin{equation}\label{e:eigen}
\mathcal{L}_\gamma(e_{j, k}^\gamma)= -\lambda_k(\gamma)(e_{j,
k}^\gamma),\quad k=0, 1, \ldots, \quad 0<\gamma\leq1\,.
\end{equation}
Moreover, $\mathcal{L}_\gamma$ is an infinitesimal generator of a
$C_0$~semigroup; denote this semigroup by
$\{\mathcal{S}_\gamma(t)\}_{t\geq 0}$\,. By the positivity
of~$-\mathcal{L}_\gamma$ we define $(-\mathcal{L}_\gamma)^{\alpha}$
for any exponent $\alpha>0$ by
 \begin{equation*}
(-\mathcal{L}_\gamma)^\alpha
u=\sum_k\big[\lambda_k(\gamma)\big]^\alpha (u_{j,k}e_{j,k}^\gamma)
\end{equation*}
for $u=(u_j)=(\sum_{k}u_{j,k}e^\gamma_{j,k})$\,. Denote by
$\mathcal{H}^\alpha_\gamma=D((-\mathcal{L}_\gamma)^\alpha)$ and
define the semi-norm~$\|\cdot\|_{\alpha, \gamma}$ in
space~$\mathcal{H}^\alpha_\gamma$ by
\begin{equation*}
\|u\|_{\alpha,
\gamma}=\|(-\mathcal{L}_\gamma)^{\alpha/2}u\|_0\,,\quad
u\in\mathcal{V}_\gamma\,.
\end{equation*}
By the same calculation as~\eqref{e:H1}, for $\alpha\in\mathbb{Z}^+\cup\{0\}$\,, \begin{equation}\label{e:alpha-norm}
\|u\|_{\alpha, \gamma}=\|u\|_\alpha\,,\quad
u\in\mathcal{H}^\alpha_\gamma\,.
\end{equation}

Given a complete probability space $(\Omega, \mathcal{F}, \{\mathcal{F}_t\}_{t\geq 0}, \mathbb{P})$, define the~$L^2(I)$ valued $Q$-Wiener process
\begin{equation*}
W(x,t)=\sum_{k=0}^\infty\sqrt{q}_k\beta_k(t)e_k(x)
\end{equation*}
where $\{\beta_k(t)\}_k$ are mutually independent standard Brownian motions and $\{e_k(x)\}_k$ is a standard basis of $L^2(I)$ with $e_0(x)=\sqrt{1/L}$ and for $k\geq 1$
\begin{eqnarray*}
e_k(x)=
  \begin{cases}
  \sqrt{\frac{2}{L}}\cos\frac{2m\pi x}{L}\,, & k=2m\,, \\
    \sqrt{\frac{2}{L}}\sin\frac{2m\pi x}{L}\,,&  k=2m-1\,.
\end{cases}
\end{eqnarray*}
Moreover, assume that the Wiener process is sufficiently well-behaved that
\begin{equation}\label{e:trace-assumption}
\sum_{k=0}^\infty k q_k<\infty\,.
\end{equation}

Now we define the $\mathcal{H}_\gamma$-valued Wiener process
$(W_j^\gamma(x, t))$ on all the overlapping elements in the
following series form
\begin{eqnarray}
W^\gamma_j(x,t)&=&
\gamma\sum_{l=0}^\infty\sqrt{q^h_{j,l}}\beta_{j,l}(t)e_{j,l}^\gamma(x),\label{e:Wj}
\end{eqnarray}
where $\{\beta_{j,l}\}_{k=0}^\infty$ are mutually independent
standard Brownian motions on $(\Omega, \mathcal{F},\mathbb{P})$ and
there are $q_{j,l}\in\R$ such that
\begin{equation*}
q_{j,l}^h=q_{j,l} h
\end{equation*}
Here for each element $j=1,\ldots, M$
\begin{equation*}
\sqrt{q^h_{j,l}}\beta_{j,l}(t)e_{j,l}^\gamma(x)=\langle W(x,t),
e_{j,l}^\gamma(x)\rangle e_{j,l}^\gamma(x)/\|e_{j,l}^\gamma\|_0\,.
\end{equation*}
Moreover, by the assumption~\eqref{e:trace-assumption}
\begin{equation}\label{e:Q-assumption}
\sum_{l=0}^\infty \lambda_l(\gamma)q_{j,l}^h<B<\infty \,,
\end{equation}
where the bound~$B$ is independent of the coupling parameter $\gamma\in (0, 1]$.
\begin{remark}\label{rem:W-gamma}
The definition of $W^\gamma(x,t)$ is similar to the definition of that in the finite element method~\cite{Thom97, Yan05, DuZhang02}.

\end{remark}

Now for fixed $h>0$ and for any $T>0$\,,
\begin{equation*}
W^\gamma(x, t)=(W^\gamma_j(x,t))\in C^{1/2}(0, T;
\mathcal{H}^2_\gamma).
\end{equation*}
and $\{W^\gamma(x,t)\}_{0<\gamma\leq 1}$ is compact in space $C(0, T; \mathcal{H}^\alpha_\gamma)$ for $\alpha<2$ for almost all $\omega\in\Omega$\,. Then for almost all $\omega\in\Omega$, the following limit
\begin{equation*}
\widetilde{W}(x,t)=(W_j(x,t))=\lim_{n\rightarrow\infty}
(W_j^{\gamma_n}(x,t))
\end{equation*}
is  well defined in space~$C(0, T; \mathcal{H})$ for some $\gamma_n\rightarrow 1$ as $n\rightarrow\infty$\,. Moreover, by the coupling conditions we have almost surely
\begin{equation*}
W_j(X_{j\pm 1}, t)=W_{j\pm 1}(X_{j\pm 1}, t), \quad j=1,2,\ldots,
M\,,
\end{equation*}
with notations $W_0(0,t)=W_{M}(X_M,t)$ and $W_{M+1}(X_{M+1},
t)=W_1(X_1, t)$\,.
\begin{remark}\label{rem:W-limit}
By the analysis on eigenfunctions $(e^\gamma_{j,k}(x))_k$ in
Section~\ref{sec:asy}, the above limit of~$W^\gamma(x,t)$ in space
$C(0, T; \mathcal{H})$ is unique in the sense of distribution for
any sequence $\gamma_n\rightarrow 1$. Further, the distribution of
$W_j(x,t)$~coincides with that of~$W_{j\pm 1}(x,t)$ on
the common overlapping domain.
\end{remark}

Finally, we explore the linear operator~$\mathcal{L}_\gamma$ as $\gamma\to 0$\,, denoted by~$\mathcal{L}_0$. Define
\begin{equation}\label{e:L-0}
\mathcal{L}_0u=\left(\frac{\p^2u_j}{\p x^2}\right),\quad
u=(u_j)\in D(\mathcal{L}_0),
\end{equation}
with the `insulating' version of the coupling conditions~\eqref{cbd1}--\eqref{cbd2}, $j=1,\ldots, M$\,,
\begin{eqnarray*}
&&u_j(X_{j\pm1})=u_j(X_j^-)=u_j(X_j^+),\\
&&u_{j,x}(X_j^+)-u_{j,x}(X_j^-)+u_{j,x}(X_{j-1})-u_{j,x}(X_{j+1})=0\,.
\end{eqnarray*}
By a standard computation~\cite{Rob08}, the spectrum of~$-\mathcal{L}_0$\,, $\{\lambda_k\}_{k=0}^\infty$\,, is
\begin{equation*}
\{0, \,\pi^2/h^2, \,4\pi^2/h^2(\text{triple}), \,9\pi^2/h^2,
\,16\pi^2/h^2(\text{triple}), \,\ldots\,, k^2\pi^2/h^2\,,
\ldots\}.
\end{equation*}
Denote the corresponding orthonormal standard eigenmodes on each
element by $\{e_{j,k}\}_{k=0}^\infty$\,, $j=1, \ldots, M$\,, then
$e_{j,0}(x)=1/\sqrt{2h}$ and $\{{e}_{j,k}(x)\}_{k\geq 1}$ are
\begin{eqnarray*}
&&\left\{ \rat{1}{\sqrt h}\sin\rat{\pi(x-X_j)}{h},\,
\left\{\rat1{\sqrt h}\cos\rat{2\pi(x-X_j)}{h}, \rat1{\sqrt
h}\sin\rat{2\pi(x-X_j)}{h}, \rat1{\sqrt
h}\sin\rat{2\pi|x-X_j|}{h}\right\}, \right.\\&&\left.\quad
\rat1{\sqrt h}\sin\rat{3\pi(x-X_j)}{h},\left\{\rat1{\sqrt h}\cos
\rat{4\pi(x-X_j)}{h}, \rat1{\sqrt h}\sin \rat{4\pi(x-X_j)}{h},
\rat1{\sqrt h}\sin\rat{4\pi|x-X_j|}{h}\right\},
\right.\\&&\quad\left. \ldots\right\}.
\end{eqnarray*}

\section{Limit system for full coupling}\label{sec:estimate}
Now we show that for full coupling, that is, as $\gamma\rightarrow
1$\,, equations~\eqref{e:uj}--\eqref{cbd2} generates a model for the
dynamics of the original physical stochastic reaction-diffusion
equation~\eqref{e:sRDe}--\eqref{e:sRDe-bd}.

This is followed by a discussion similar to the case of Dirichlet
boundary conditions~\cite{WR09}, here we just state the result and
omit the detailed proof.

\begin{theorem}
Assume bound~\eqref{e:Q-assumption} and
$u^\gamma(0)\in\mathcal{H}^2_\gamma$ with $\|u^\gamma(0)\|_2\leq
C_0$ which is independent of $\gamma$ and $u^\gamma(0)\rightarrow
u^0$ in $\mathcal{H}$ as~$\gamma\rightarrow 1$\,. Then for
any~$T>0$\,,~$u^\gamma$ converges to $u$ in distribution as
$\gamma\rightarrow 1$ in space $C(0, T; \mathcal{H})$
where~$u=(u_j)$ solves
\begin{equation}\label{e:gamma-1}
\p_t u_j=\p_{xx}u_j+\alpha u_j^3+\alpha f(u_j)+\sigma\p_tW_j\quad
\text{on}\quad I_j
\end{equation}
with  $u(0)=u^0$ and
\begin{equation}\label{e:gamma-1bc}
u_j(X_{j\pm1}, t)=u_{j\pm 1}(X_{j\pm 1}, t)\,, \quad
u_{j}(x,t)=u_{j\pm M}(x,t)\,.
\end{equation}
\end{theorem}

To prove the above result needs some  energy estimates on the
solutions $u^\gamma(x,t)$.
 By the definition
of~$\mathcal{L}_\gamma$, the \spde{}s~\eqref{e:uj}--\eqref{cbd2}
takes the following abstract from
\begin{align} &
du^\gamma(t)=[\mathcal{L}_\gamma u^\gamma(t)+\alpha\gamma^2
u^\gamma(t)+\alpha F(u^\gamma(t))]\,dt+\sigma dW^\gamma(t)\,,
\nonumber\\& u^\gamma(0)=(u_j^\gamma(0))\,, \label{e:abs-u-gamma}
\end{align}
where $u^\gamma(t)=(u_j^\gamma(t))$ and
$F(u^\gamma)=(f(u_j^\gamma))$. Writing the temporal dependence
explicitly, in a mild sense we have
\begin{align*}
u^\gamma(t)={}& \mathcal{S}_\gamma(t)u^\gamma(0)+
\alpha\int_0^t\mathcal{S}_\gamma(t-s)\left[\gamma^2u^\gamma(s)+F(u^\gamma(s))\right]ds
\\&{}
+\sigma\int_0^t\mathcal{S}_\gamma(t-s) \, dW^\gamma(s)\,.
\end{align*}
Then by a standard semigroup approach~\cite{PZ92}, for any
$u^\gamma(0)\in\mathcal{H}_\gamma$ and any $T>0$\,,
\eqref{e:abs-u-gamma}~has a unique mild solution $u^\gamma(t)\in
C(0, T; \mathcal{H}_\gamma)\cap L^2(0, T; \mathcal{V}_\gamma)$\,.
Letting
\begin{equation*}
z^\gamma(t)=\sigma\int_0^t\mathcal{S}_\gamma(t-s) \, dW^\gamma(s)\,,
\end{equation*}
we have the following lemma.
\begin{lemma}
Assume boundedness~\eqref{e:Q-assumption}. For any $T>0$ and $q>0$
there is a positive constant~$C_q(T)$ such that
\begin{equation*}
\mathbb{E}\sup_{0\leq t\leq T}\|z^\gamma(t)\|_{2, \gamma}^q\leq
C_q(T)\,.
\end{equation*}
\end{lemma}
\begin{proof}
Since $-\mathcal{L}_\gamma$ is positive and self-adjoint,
\begin{equation*}
\|\mathcal{S}_\gamma(t)\|_{\mathcal{L}}\leq 1\,, \quad t\geq 0\,.
\end{equation*}
Then the result follows from the stochastic factorization
formula~\cite{PZ92}.
\end{proof}

Now define the difference $w^\gamma=u^\gamma-z^\gamma$\,, then
\begin{equation*}
dw^\gamma(t)=[\mathcal{L}_\gamma
w^\gamma(t)+\alpha\gamma^2w^\gamma(t)+\alpha F(u^\gamma)]dt\,,\quad
w^\gamma(0)=u^\gamma(0)\,.
\end{equation*}
By the standard energy estimate  to stochastic reaction-diffusion
equations with more general nonlinearity~\cite{WR08} we have
\begin{lemma}\label{lem:est}
Assume $u^\gamma(0)\in\mathcal{H}_\gamma$\,,
 then for any $T>0$\,, there is a positive constant $C_T>0$ such that for any
$p\in \mathbb{Z}^+$
\begin{equation*}
\mathbb{E}\|u^\gamma\|_{C(0,T; \mathcal{H}_\gamma)\cap L^p(0, T;
\mathcal{V}_\gamma)}+\mathbb{E}\|\p_tu^\gamma\|_{L^2(0, T;
\mathcal{H}^{-1})}\leq C_T(1+\|u^\gamma(0)\|_0^2)\,.
\end{equation*}
\end{lemma}
We show that $\{\mathcal{D}(u^\gamma)\}_\gamma$, the distribution
of~$u^\gamma$ in space $C(0,T; \mathcal{H})$, is tight. For this we
need the following lemma by Simon~\cite{Sim87}.
\begin{lemma}\label{lem:compemb}
Assume~$E$, $E_0$ and~$E_1$ are Banach spaces such that $E_1\Subset
E_0$\,, the interpolation space $(E_0, E_1)_{\theta,1}\subset E$
with $\theta\in (0, 1)$  and $E\subset E_0$ with
$\subset$~and~$\Subset$ denoting continuous and compact embedding
respectively. Suppose $p_0,p_1\in [1,\infty]$ and $T>0$\,, such that
\begin{equation*}
X  \text{ is a bounded set in } L^{p_1}(0, T; E_1)
\end{equation*}
and
\begin{equation*}
\p X:=\{\p v: v\in X\} \text{ is a bounded set in } L^{p_0}(0, T;
E_0).
\end{equation*}
Here $\p$ denotes the distributional derivative. If
$1-\theta>1/p_\theta$ with
 \begin{equation*}
\frac{1}{p_\theta}=\frac{1-\theta}{p_0}+\frac{\theta}{p_1}\,,
 \end{equation*}
then $X$~is relatively compact in $C(0, T; E)$.
\end{lemma}
By the above lemma, and noticing the relation~\eqref{e:alpha-norm},
we have the following theorem.

\begin{theorem}\label{thm:tight}
Assume~\eqref{e:Q-assumption} and $u^\gamma(0)\in\mathcal{H}_\gamma$
with $\|u^\gamma(0)\|_0\leq C_0$ which is independent of~$\gamma$.
For any $T>0$\,, $\mathcal{D}(u^\gamma)$ is tight
in~$C(0,T;\mathcal{H})$\,.
\end{theorem}

Similarly if $u^\gamma(0)\in\mathcal{H}_\gamma^2$ with
$\|u^\gamma(0)\|_{2, \gamma}\leq C_0$ which is independent
of~$\gamma$, for any $T>0$ there is a positive constant $C_T>0$ such
that
\begin{equation*}
\mathbb{E}\sup_{0\leq t\leq T}\|u^\gamma(t)\|_{2,\gamma}\leq C_T\,.
\end{equation*}
Then by the embedding of $H^2(I)\subset C^1(I)$~\cite{Sob},
\begin{equation}\label{e:est-int-bd}
\mathbb{E}\left(\left|\frac{\p u_j^\gamma}{\p x}(X_{j\pm 1},t)\right|\right)\leq C_T\,,\quad 0\leq t\leq T\,.
\end{equation}
By the above estimates we can treat the boundary value in passing to
the limit $\gamma\rightarrow 1$ of full coupling.

By Theorem~\ref{thm:tight}, for any $\kappa>0$ there is a compact
set $K_\kappa\subset C(0, T; \mathcal{H})$ such that
\begin{equation*}
\mathbb{P}\{u^\gamma\in K_\kappa \}\geq 1-\kappa\,.
\end{equation*}
Then there is a function $u\in C(0, T; \mathcal{H})$ and a
subsequence $\gamma_n\rightarrow 1$ as $n\rightarrow\infty$\,, such
that in probability
\begin{equation*}
u^{\gamma_n}\rightarrow u \quad\text{as } n\rightarrow\infty\,.
\end{equation*}
Now we determine the equation solved by the limit~$u$. Define a test
function $\varphi\in C_0^\infty(0, L)$ and define
\begin{equation*}
\varphi_j=\varphi|_{I_j}\,.
\end{equation*}
Then by the boundary conditions we have in the variational form for
the system~\eqref{e:uj}--\eqref{cbd2}
\begin{eqnarray}
&&\langle u^{\gamma_n}(t), \varphi\rangle
\nonumber\\&=&\langle
u^{\gamma_n}(0), \varphi\rangle+\int_0^t\langle
\mathcal{L}_{\gamma_n} u^{\gamma_n}(s)+\alpha\gamma_n^2
u^{\gamma_n}(s),\varphi\rangle\,ds \nonumber\\&&{}
+\int_0^t\langle\alpha F(u^{\gamma_n}(s)), \varphi\rangle\,ds
+\int_0^t\langle\sigma\,dW^{\gamma_n}(s),\varphi\rangle\nonumber\\
&=&\langle u^{\gamma_n}(0),
\varphi\rangle-\sum_j\int_0^t\int_{X_{j-1}}^{X_{j+1}}\frac{\p
u^{\gamma_n}_j(s)}{\p x} \frac{\p\varphi_j}{\p x}\,dx\, ds
\nonumber\\&&{}+
\sigma\sum_j\int_0^t\int_{X_{j-1}}^{X_{j+1}}\,dW^{\gamma_n}_j(s)\varphi_j\,dx\nonumber\\
&&{}+ \alpha
\gamma_n^2\sum_j\int_0^t\int_{X_{j-1}}^{X_{j+1}}u^{\gamma_n}_j(s)
\varphi_j\,dx\,ds+ \alpha\sum_j\int_0^t\int_{X_{j-1}}^{X_{j+1}}
f(u^{\gamma_n}_j(s)) \varphi_j\,dx\, ds\nonumber
\\&&{} -\gamma'_n\sum_j\int_0^t\left[\frac{\p u^{\gamma_n}_{j-1}(X_j,
s)}{\p x}-\frac{\p u^{\gamma_n}_{j+1}(X_j, s)}{\p
x}\right]\varphi_j(X_j)\,ds\nonumber\\&&{}+
\gamma_n'\sum_j\left[\frac{\p u^{\gamma_n}_j(X_{j-1}, s)}{\p
x}-\frac{\p u^{\gamma_n}_j(X_{j+1}, s)}{\p
x}\right]\varphi_j(X_j)\,ds\,.\label{e:variation form}
\end{eqnarray}
Then  by estimate~\eqref{e:est-int-bd},  letting
$n\rightarrow\infty$\,, that is $\gamma_n'\rightarrow 0$\,, the last
two terms disappear. Notice that $f(u^{\gamma_n}_j)\rightarrow
f(u_j)$ weakly in space~$L^2(0,T; L^2)$ and by the assumption
on~$W^{\gamma_n}$ we have, by passing to the limit
$n\rightarrow\infty$\,,
\begin{eqnarray}
\langle u(t), \varphi\rangle &=&\langle u(0),
\varphi\rangle-\sum_j\int_0^t\int_{X_{j-1}}^{X_{j+1}}  \frac{\p
u_j(s)}{\p x}
\frac{\p\varphi_j}{\p x}\,dx\,ds\nonumber\\
&&{}+\alpha\sum_j\int_0^t\int_{X_{j-1}}^{X_{j+1}}
(u_j(s)-u_j^3(s))\varphi\,dx\,ds
\nonumber\\&&+\int_0^t\langle\sigma\,d\widetilde{W}(s),\varphi\rangle\label{e:limit
1}
\end{eqnarray}
with $\widetilde{W}(t)=(W_j(t))$ which is well defined by
Remark~\ref{rem:W-limit}. Then a density argument yields that
$u=(u_j)$ solves the following stochastic equations
\begin{eqnarray*}
\p_tu_j&=&[\p_{xx} u_j+\alpha u_j-\alpha (u_j)^3]\,dt+\sigma\,\p_tW_j \quad \text{on }I_j,\label{e:limit-1-u1}
\end{eqnarray*}
with coupling boundary conditions
\begin{equation*}
u_j(X_{j\pm1}, t)=u_{j\pm 1}(X_{j\pm 1}, t)\,, \quad
u_{j}(x,t)=u_{j\pm M}(x\pm L,t)\,.
\end{equation*}

\begin{remark}
By the boundary condition~\eqref{e:gamma-1bc} and Remark~\ref{rem:W-limit}, the distributions of  $u_j$ and $u_{j+1}$ in space $C(0, T; L^2(X_j, X_j+h))$ coincides.

\end{remark}

 Now define
\begin{equation*}
u(x,t)=u_j(x,t),\quad x\in [X_j, X_j+h]
\end{equation*}
and an $L^2(I)$-valued Wiener process~$\overline{W}(x,t)$ as
\begin{equation*}
\overline{W}(x,t)=W_j(x,t), \quad x\in[X_j, X_j+h].
\end{equation*}
Then~$u(x,t)$~solves the stochastic reaction-diffusion equation~\eqref{e:sRDe}--\eqref{e:sRDe-bd} with the noise term $W(t)$ replaced by $\overline{W}(t)$ without changing the distribution. So~\eqref{e:uj}--\eqref{cbd2} recovers the original system~\eqref{e:sRDe}--\eqref{e:sRDe-bd}, in distribution, in the limit of full coupling, as $\gamma\rightarrow 1$\,.

%

\section{Amplitudes on the elements}\label{sec:asy}
We derive a discrete  macroscopic approximation
to the system of \spde{}s~\eqref{e:uj}--\eqref{cbd2} based upon small coupling parameter
$\gamma>0$\,. By the analysis on operator~$\mathcal{L}_0$ in
section~\ref{sec:coulping BDC}, for $\gamma=0$ the dominant mode
is~$(e_{j0})$, so by hyperbolicity we expect that for small
$\gamma>0$, the dominant mode is~$(e_{j0}^\gamma)$. This is followed
by the analysis on the continuity of $\{(e^\gamma_{jk})\}_k$ and
$\lambda_k(\gamma)$ on coupling parameter~$\gamma$. Further the
asymptotic expansion for $(e_{j0}^\gamma)$ in $\gamma$   shows that
the grid value approximates the amplitude on each
element.

For this we study the continuity properties of~$\mathcal{L}_\gamma$ as coupling $\gamma\rightarrow 0$\,.  We use variational convergence for operators~\cite{Att}. For any subsequence~$\gamma_n$ with $\gamma_n\rightarrow 0$ as $n\rightarrow\infty$\,, we introduce the G-convergence for~$\mathcal{L}_{\gamma_n}$.
\begin{definition}[\textbf{G-convergence}]
Operator~$\mathcal{L}_{\gamma_n}$ is said to be graph-convergent (G-convergent) to~$\mathcal{L}_0$ as $n\rightarrow\infty$ if for every~$(u, v)$ with $v=\mathcal{L}_0u$, there exists a sequence~$(u^n, v^n)$ with $v^n=\mathcal{L}_{\gamma_n}u^n$ such that $u^n\rightarrow u$ strongly in~$\mathcal{V}$ and $v^n\rightarrow v$ strongly in~$\mathcal{V}^*$, the dual space of~$\mathcal{V}$.
\end{definition}

Now for any $u=(u_j(x))\in \mathcal{V}_0$, denote by $v=\mathcal{L}_0u$\,. First we choose bounded set $\{v^n\}\subset\mathcal{H}_{\gamma_n}$ such that $v^n\rightarrow v$ in the dual space~$\mathcal{V}^*$. Then solve the following equation
\begin{equation}\label{e:un}
\mathcal{L}_{\gamma_n}u^n=v^n.
\end{equation}
By the relation~\eqref{e:alpha-norm}, $\{u^n\}$~is bounded
in~$\mathcal{H}^2$ which yields that $\{u^n\}$~is compact
in~$\mathcal{V}$. Then there is a subsequence, which we still denote
by $\{u^n\}$, that converges to~$\tilde{u}$ in $\mathcal{V}$ as
$n\rightarrow\infty$\,. Multiplying testing function $\varphi\in
C_0^\infty(0, L)$ on both sides of~\eqref{e:un} and passing to the
limit $n\rightarrow\infty$\,, we have
\begin{equation}\label{e:un0}
\mathcal{L}_0\tilde{u}=v
\end{equation}
which yields that $u=\tilde{u}$ by the uniqueness of the solution to~\eqref{e:un0}. Then we have $\mathcal{L}_{\gamma_n}$ is G-convergent to~$\mathcal{L}_0$\,.

Now we draw the following result on the continuity of eigenvalues
and eigenfunctions in coupling~$\gamma$~\cite{WR09}.
\begin{theorem}\label{thm:cov}
\begin{equation*}
\lim_{\gamma\rightarrow
0}\lambda_k(\gamma)=\lambda_k=-k^2\pi^2/h^2,\quad k=0,1,2, \ldots\,.
\end{equation*}
Let $m_k$ be the multiplicity of~$\lambda_k(\gamma)$, the sequence
of subspaces~$\mathbb{L}^{\gamma}_k$ of dimension~$m_k$ generated by
$((e_{jk}^{\gamma,1}), \ldots, (e_{jk}^{\gamma,m_k}))$ converges
in~$\mathcal{H}$ to the eigenspace of~$\mathcal{L}_0$ corresponding
to $\lambda_k=-k^2\pi^2/h^2$.
\end{theorem}
\begin{remark}
Such convergence of the eigenspaces is called
Mosco-con\-ver\-gence~\cite{Att}. In the sense of this convergence,
$(e_{j,k}^{\gamma, l})$ may converges to $(e_{j,k}^{l'})$ instead of
$(e_{j,k}^{l})$\,, $1\leq l\neq l'\leq m_k$\,.
\end{remark}
By the above result for small coupling $\gamma>0$\,, we  write for each element $j=1,\ldots, M$\,,
\begin{equation*}
e^\gamma_{j,0}(x)=e_{j,0}(x)+\tilde{e}^\gamma_{j,0}(x)
\end{equation*}
with
\begin{equation*}
\tilde{e}^\gamma_{j,0}(x)\rightarrow 0  \quad\text{as }
\gamma\rightarrow 0\,.
\end{equation*}
Next we show that $\tilde{e}_{j,0}^\gamma(x)$ has an expansion to
$\gamma^2$ terms in the coupling parameters:
\begin{equation*}
\frac{\p^2 }{\p
x^2}{\tilde{e}}_{j,0}^\gamma(x)=\lambda_0(\gamma){\tilde{e}}_{j,0}^\gamma(x)+\lambda_0(\gamma)/\sqrt{2h}\,,
\end{equation*}
with boundary condition (\ref{cbd1})--(\ref{cbd2})\,.  Now if we
account for the coupling and boundary condition by
\begin{align*}
\tilde{e}_{j,0}^\gamma(x)=&
\tilde{e}_{j,0}^\gamma(X_j)+\bar{e}_{j,0}^\gamma(x)
  \\&{}
  +\begin{cases}
  \frac{\gamma}{h}\left[e_{j,0}^\gamma(X_j)-e_{j-1,0}^\gamma(X_{j-1})\right](x-X_j)\,,& X_{j-1}\leq x\leq
X_j\,,\\
\frac{\gamma}{h}\left[e_{j+1,0}^\gamma(X_{j+1})-e_{j,0}^\gamma(X_j)\right](x-X_j)\,,&
X_j\leq x\leq X_{j+1}\,.
  \end{cases}
\end{align*}
Then $\bar{e}_{j,0}(x)$ solves
\begin{equation}\label{e:bar-e}
\frac{\p^2 }{\p
x^2}{\bar{e}}_{j,0}^\gamma(x)=\lambda_0(\gamma){\tilde{e}}_{j,0}^\gamma(x)+\lambda_0(\gamma)/\sqrt{2h}
\end{equation}
with boundary condition
\begin{equation*}
\bar{e}_{j,0}(X_{j\pm 1})=\bar{e}_{j,0}(X_{j})=0\,.
\end{equation*}
But, by~\eqref{cbd2} we also have coupling in the derivatives:
\begin{eqnarray*}
&&\p_x\bar{e}_{j,0}^\gamma(X_j^-)-\p_x\bar{e}_{j,0}^\gamma(X_j^+)+\gamma
\p_x\bar{e}_{j-1,0}^\gamma(X_j)-\gamma \p_x\bar{e}_{j+1,0}^\gamma(X_j)\nonumber\\
&&\quad{}-\gamma'\p_x\bar{e}_{j,0}^\gamma(X_{j-1})+\gamma'\p_x\bar{e}_{j,0}^\gamma(X_{j+1})\\{}
&=&
2\frac{\gamma^2}{h}\left[e_{j+1,0}^\gamma(X_{j+1})-2e_{j,0}^\gamma(X_j)+e_{j-1,0}^\gamma(X_{j-1})
\right]\,, \quad j=1, 2,\ldots, M\,,
\end{eqnarray*}
Then by the above boundary condition, equation (\ref{e:bar-e}) and
the fact $\tilde{e}_{j,0}^\gamma(x)\rightarrow 0$ as
$\gamma\rightarrow 0$\,, we have
\begin{equation*}
\lambda_0(\gamma)=\mathcal{O}(\gamma^2) \quad\text{as }
\gamma\rightarrow 0
\end{equation*}
which implies that $\tilde{e}^\gamma_{j0}(x)$ have an expansion to
$\gamma^2$~terms in the coupling parameter.

Assume we have the following asymptotic expansion for each element,
$j=1,\ldots, M$\,,
\begin{equation}\label{e:ej0}
e_{j,0}^\gamma(x)=e_{j,0}^\gamma(X_j)+\gamma
F^\gamma_{j,1}(x)+\gamma^2F^\gamma_{j,2}(x)+F_{j,3}^\gamma(x)
\end{equation}
where $F_{j,3}^\gamma(x)=\Ord{\gamma^3}$\,. By $\lambda_0(0)=0$ and
the coupling boundary condition~(\ref{cbd1})--(\ref{cbd2})\,,
$F^\gamma_{j, k}$\,, $k=1,2$\,, are $k$th order polynomial in $x$\,.
Then also by the boundary condition (\ref{cbd1})--(\ref{cbd2}) we
have
\begin{eqnarray}
F^\gamma_{j,1}(x)=
  \begin{cases}
   \frac{e_{j,0}^\gamma(X_j)-e_{j-1,0}^\gamma(X_{j-1})}{h}(x-X_j),&
X_{j-1}\leq x\leq X_j\,,\\
    \frac{e_{j+1,0}^\gamma(X_{j+1})-e_{j,0}^\gamma(X_j)}{h}(x-X_j),&
X_j\leq x\leq X_{j+1}\,,
\end{cases}
\label{e:Fj1}
\end{eqnarray}
and
\begin{eqnarray}
F^\gamma_{j,2}(x)=
  \begin{cases}
   A_j(x-X_j)(x-X_{j-1}),&
X_{j-1}\leq x\leq X_j\,,\\
    A_j(x-X_j)(x-X_{j+1}),&
X_j\leq x\leq X_{j+1}\,,
\end{cases}
\label{e:Fj2}
\end{eqnarray}
with
\begin{equation*}
A_j=\frac{e_{j-1,
0}^\gamma(X_{j-1})-2e_{j,0}^\gamma(X_j)+e_{j+1,0}^\gamma(X_{j+1})}{2h^2}\,.
\end{equation*}
The above asymptotic expansion shows that for small coupling $\gamma>0$\,, the first mode is dominating and the grid value~$u_j^\gamma(X_j, t)$ approximates the amplitude of the field~$u^\gamma(x,t)$ on the element~$I_j$.

\section{Macroscopic  models for small coupling}
\label{sec:macr-model} By the asymptotic expansion in the previous section~\ref{sec:asy} we derive a discrete macroscopic approximation model to~\eqref{e:uj}--\eqref{cbd2} for small coupling $\gamma>0$\,. For this we first apply an averaging method to reduce~\eqref{e:uj}--\eqref{cbd2} onto the slow mode~$(e_{j0}^\gamma)$.

We split $(u_j^\gamma)$ into slow part and fast part. Define map $P_0^\gamma$ on $\mathcal{H}_\gamma$ to $\mathcal{H}$
\begin{equation*}
P^\gamma_0(u_j^\gamma)=\left(\langle u_j^\gamma, e_{j,0}^\gamma
\rangle e_{j,0}^\gamma(x)/\|e_{j,0}^\gamma\|_0^2\right)\quad
\text{and}\quad
 P^\gamma_1=I-P^\gamma_0 \,,
\end{equation*}
where $I$~is the identity operator on~$\mathcal{H}_\gamma$. And for no coupling, $\gamma=0$\,, write $P_0=P_0^0$ and $P_1=P_1^0$\,. Then denote by $(u^\gamma_j(x,t))$ the solution to system~\eqref{e:uj}--\eqref{cbd2} and make the following expansion
\begin{equation*}
(u_j^\gamma(x,t))=\sum_{k=0}^\infty
a_k^\gamma(t)(e_{j,k}^\gamma(x)).
\end{equation*}
Now define the slow part and fast part, respectively,
\begin{equation*}
U^\gamma(x,t)=(a_0^\gamma(t)e_{j,0}^\gamma(x)) \quad\text{and}\quad
V^\gamma(x,t)=(u_j^\gamma(x,t))-U^\gamma(x,t).
\end{equation*}
Then we have that these satisfy the coupled \spde{}s
\begin{eqnarray}
dU^\gamma&=&\Big[\mathcal{L}_\gamma U^\gamma+\alpha\gamma^2
U^\gamma+\alpha P^\gamma_0F(U^\gamma, V^\gamma)\Big]dt
 +\sigma\gamma dB_0^\gamma\,,\label{e:abs-U1}\\
dV^\gamma&=&\Big[\mathcal{L}_\gamma
 V^\gamma+\alpha\gamma^2 V^\gamma+\alpha P^\gamma_1F(U^\gamma,
V^\gamma)\Big]dt+\sigma\gamma dB_1^\gamma\,,\label{e:abs-U2}
\end{eqnarray}
where
\begin{equation*}
B_0^\gamma(t)=\Big(\sqrt{q^h_{j,0}}\beta_{j,0}(t)\Big)
\quad\text{and}\quad
B_1^\gamma(t)=\Big(\sum_{k=1}^\infty\sqrt{q^h_{j,k}}\beta_{j,k}(t)e_{j,k}^\gamma(x)\Big).
\end{equation*}
By the analysis of Section~\ref{sec:asy}, $\lambda_0(\gamma)=
\Ord{\gamma^2}$ as $\gamma\rightarrow 0$\,, then for small coupling~$\gamma>0$\,, \eqref{e:abs-U1}--\eqref{e:abs-U2} have completely separated time scales. Thus an averaging approach applies to derive a macroscopic reduced system over the time scale~$\gamma^{-2}T$ for any $T>0$~\cite{WR08}. For this introduce a slow time scale $t'=\gamma^2t$ and small fields
\begin{equation}\label{e:small fields}
U^\gamma(t)=\gamma \widetilde{U}^\gamma(\gamma^2 t)
\quad\text{and}\quad
V^\gamma(t)=\gamma \widetilde{V}^\gamma(\gamma^2 t),
\end{equation}
then on the slow time scale~$t'$
\begin{eqnarray*}
d\widetilde{U}^\gamma(t')&=&\left[\gamma^{-2}\mathcal{L}_\gamma
\widetilde{U}^\gamma+\alpha\widetilde{U}^\gamma
  +\alpha P^\gamma_0F(\widetilde{U}^\gamma, \widetilde{V}^\gamma)\right]dt'+\sigma\gamma^{-1} d\widetilde{B}_0^\gamma(t'),\\
d\widetilde{V}^\gamma(t')&=&\left[\gamma^{-2}\mathcal{L}_\gamma
\widetilde{V}^\gamma+\alpha\widetilde{V}^\gamma +\alpha
P^\gamma_1F(\widetilde{U}^\gamma,
\widetilde{V}^\gamma)\right]dt'+\sigma\gamma^{-1}
d\widetilde{B}_1^\gamma(t').
\end{eqnarray*}
Here $\widetilde{B}_0^\gamma(t')=\gamma B_0^\gamma(\gamma^{-2}t')$ and $\widetilde{B}_1^\gamma(t')=\gamma B_1^\gamma(\gamma^{-2}t')$ are Wiener processes with the same distributions as those of $B_0^\gamma(t')$~and~$B_1^\gamma(t')$, respectively, due to the scaling properties of the Wiener process.

Let $\tilde{\eta}^\gamma(t')=(\tilde{\eta}_{j0}^\gamma(t'))\in P^\gamma_1 \mathcal{H}_\gamma$ be the unique stationary solution of the following linear equation
\begin{equation}\label{e:tilde-eta}
d\tilde{\eta}^\gamma(t')=\gamma^{-2}\mathcal{L}_\gamma\tilde{\eta}^\gamma(t')\,dt'+\sigma\gamma^{-1}\,d\widetilde{B}^\gamma_1(t').
\end{equation}
Then by an energy estimate and almost the same discussion to that by Wang and Roberts~\cite{WR08}, for any $T\geq 0$\,, there is a positive constant~$C_T$ such that
\begin{equation}\label{e:appr-high}
\sup_{0\leq t'\leq
T}\mathbb{E}\|\widetilde{V}^\gamma(t')-\tilde{\eta}^\gamma(t')\|_0\leq
\gamma^2C_T\left(\mathbb{E}\|\tilde{\eta}^\gamma(0)\|_0+\|(u^\gamma_{j0}(0))\|_0^6\right).
\end{equation}
We have no explicit expressions of $(e_{jk}^\gamma(x))$\,, $k\geq 1$\,, so for our purpose we derive another approximation for~$\widetilde{V}^\gamma(t')$.
\begin{lemma}
Assume bound~\eqref{e:Q-assumption}. For any $T>0$\,, there is positive constant~$C_T$ such that
\begin{equation*}
\sup_{0\leq t'\leq
T}\mathbb{E}\|\widetilde{V}^\gamma(t')-\tilde{\eta}(t')\|_0\leq
\gamma
C_T\left(\mathbb{E}\|\tilde{\eta}^\gamma(0)\|_0+\mathbb{E}\|\tilde{\eta}(0)\|_0+\|(u_{j0}(0))\|_0^6\right),
\end{equation*}
where $\tilde{\eta}(t')=(\tilde{\eta}_{j0}(t'))$ is the unique stationary solution of the following linear equation
\begin{equation*}
d\tilde{\eta}(t')=\gamma^{-2}\mathcal{L}_0\tilde{\eta}(t')\,dt'+\sigma\gamma^{-1}\,d\widetilde{B}_1(t')
\end{equation*}
and
\begin{equation*}
\widetilde{B}_1(t')=\sum_{k=1}^\infty\Big(\gamma\sqrt{q^h_{j,k}}\beta_{j,k}(\gamma^{-2}t')e_{j,k}(x)\Big)
\end{equation*}
with distributions independent of coupling parameter~$\gamma$.
\end{lemma}

\begin{proof}
Expand $\tilde{\eta}^\gamma(t)$ and~$\tilde{\eta}(t)$ by the eigenfunctions of~$\mathcal{L}_\gamma$ and~$\mathcal{L}_0$ respectively as
\begin{equation*}
\tilde{\eta}^\gamma(x, t')=\sum_{k=1}^\infty\tilde{\eta}^\gamma_k(x,
t')
\quad\text{and}\quad \tilde{\eta}(x, t')=\sum_{k=1}^\infty\tilde{\eta}_k(x,
t'),
\end{equation*}
where
\begin{eqnarray*}
\tilde{\eta}^\gamma_k(x, t')&=&\left\langle \tilde{\eta}^\gamma(t'),
(e_{j,k}^\gamma(x)) \right\rangle
(e_{j,k}^\gamma(x))/\|(e_{j,k}^\gamma)\|_0^2\,,\\
\tilde{\eta}_k(x, t')&=&\big\langle \tilde{\eta}(t'), (e_{j,k}(x))
\big\rangle(e_{j,k}(x))/M\,.
\end{eqnarray*}
Then for $k\geq 1$
\begin{equation*}
d\tilde{\eta}_k^\gamma(t')=-\frac{1}{\gamma^2}\lambda_k(\gamma)\tilde{\eta}^\gamma_k(t')dt'
+\frac{\sigma}{\gamma} d\widetilde{B}^\gamma_{1,k}(t')
\end{equation*}
and
\begin{equation*}
d\tilde{\eta}_k(t')=-\frac{1}{\gamma^2}\lambda_k\tilde{\eta}_k(t')dt'
+\frac{\sigma}{\gamma} d\widetilde{B}_{1,k}(t')
\end{equation*}
with
\begin{eqnarray*}
\widetilde{B}_{1,k}^\gamma(t')&=&\gamma\Big(\sqrt{q^h_{j,k}}\beta_{j,k}(\gamma^{-2}t')e^\gamma_{j,k}(x)\Big),\\
\widetilde{B}_{1,k}(t')&=&\gamma\Big(\sqrt{q^h_{j,k}}\beta_{j,k}(\gamma^{-2}t')e_{j,k}(x)\Big).
\end{eqnarray*}
Using the It\^o formula and the stationary property of $\tilde{\eta}^\gamma_k$ and $\tilde{\eta}_k$\,, there is a positive constant~$C$ such that for $k\geq 1$
\begin{equation}\label{e:est-etak}
\mathbb{E}\|\tilde{\eta}^\gamma_k(t')\|_0\leq C \quad\text{and}\quad
\mathbb{E}\|\tilde{\eta}_k(t')\|_0\leq C\,.
\end{equation}

Define the difference $z_k^\gamma(t')=\tilde{\eta}^\gamma_k(t')-\tilde{\eta}_k(t')$ which solves
\begin{equation*}
dz^\gamma_k(t')=-\frac{1}{\gamma^2}\left[\lambda_kz^\gamma_k(t')+
(\lambda_k(\gamma)-\lambda_k)\tilde{\eta}^\gamma_k(t')\right]dt'
+\frac{\sigma}{\gamma}d\left[\widetilde{B}_{1,k}^\gamma(t')-\widetilde{B}_{1,k}(t')\right],
\end{equation*}
and hence
\begin{eqnarray*}
z^\gamma_k(t')&=&e^{-\lambda_kt'/\gamma^2}z^\gamma_k(0)+\frac{\lambda_k-\lambda_k(\gamma)}{\gamma^2}
\int_0^{t'}e^{-\lambda_k(t'-s)/\gamma^2}\tilde{\eta}_k^\gamma(s)\,ds\\
&&{}+\frac{\sigma}{\gamma}\int_0^{t'}e^{-\lambda_k(t'-s)/\gamma^2}d[\widetilde{B}_{1,k}^\gamma(s)-\widetilde{B}_{1,k}(s)].
\end{eqnarray*}
Then by the analysis of Section~\ref{sec:asy} on~$\lambda_k(\gamma)$, $k\geq 1$\,, assumption~\eqref{e:Q-assumption} and the estimates~\eqref{e:est-etak}  we have
\begin{equation*}
\sup_{0\leq t'\leq
T}\mathbb{E}\|\tilde{\eta}^\gamma(t')-\tilde{\eta}(t')\|_0\leq
\gamma
C_T\mathbb{E}\left(\|\tilde{\eta}^\gamma(0)\|_0+\|\tilde{\eta}(0)\|_0\right).
\end{equation*}
Thus by~\eqref{e:appr-high}, the proof is complete.
\end{proof}

Now by~\eqref{e:small fields} and the above result,  for~$V^\gamma$ on the original time scale,
\begin{equation}\label{e:appr-fast}
\sup_{0\leq t\leq
\gamma^{-2}T}\mathbb{E}\|V^\gamma(t)-\gamma\tilde{\eta}(\gamma^2t)\|_0
\leq
\gamma^2
C_T\left(\mathbb{E}\|\tilde{\eta}^\gamma(0)\|_0+\mathbb{E}\|\tilde{\eta}(0)\|_0+\|(u_{j0}(0))\|_0^6\right).
\end{equation}
 Moreover,
\begin{equation}\label{e:E-eta}
\mathbb{E}{\tilde{\eta}^2}(t)=\sigma^2\sum_{k=1}^\infty\frac{1}{2\lambda_k}(q^h_{j,k}e^2_{j,k}(x)).
\end{equation}

For~$\widetilde{U}^\gamma$ we follow an averaging approach~\cite{WR08} which yields the following averaged equation
\begin{equation}\label{e:averged}
d\bar{U}^\gamma(t')=\left[\gamma^{-2}\mathcal{L}_\gamma
\bar{U}^\gamma(t')+\alpha\bar{U}^\gamma(t')+\alpha
P^\gamma_0\bar{F}(\bar{U}^\gamma(t'))\right]dt'
+\sigma\gamma^{-1}\,d\widetilde{B}_0^\gamma(t')
\end{equation}
where $\bar{U}^\gamma=(\bar{u}_{j0}^\gamma)$ and
\begin{equation*}
\bar{F}(\cdot)=
\mathbb{E}\left[F(\cdot,\tilde{\eta}(\omega))\right].
\end{equation*}
By the definition of~$f$ and  that~$\tilde{\eta}$ is Gaussian with zero mean, we have
\begin{equation*}
\bar{F}(\bar{U}^\gamma) = - \left(\left(\bar{u}^\gamma_{j0}\right)^3
+3\bar{u}_{j0}^\gamma\mathbb{E}\tilde{\eta}^2_{j0} \right).
\end{equation*}

Moreover, by a deviation argument~\cite{WR08, WR08-1},  stochastic effects in these subgrid scale fast modes are fed into the slow modes by the nonlinear interaction. So we have the following averaged equation plus deviation
\begin{eqnarray}\label{e:averged+deviation}
d\bar{U}^\gamma(t')&=&\left[\gamma^{-2}\mathcal{L}_\gamma
\bar{U}^\gamma
+\alpha\bar{U}^\gamma+\alpha
P^\gamma_0\bar{F}(\bar{U}^\gamma)\right]dt'
\nonumber\\&&{}
+\sigma\gamma^{-1}\,d\widetilde{B}_0^\gamma(t')+\alpha\gamma
\sqrt{\bar{Q}(\bar{U}^\gamma)}d\bar{\beta}(t')
\end{eqnarray}
with, for fixed~$\bar{U}^\gamma$,
\begin{align*}
\bar{Q}(\bar{U}^\gamma)=2\mathbb{E}\int_0^\infty
&
P_0^\gamma\big[
F(\bar{U}^\gamma+\tilde{\eta}(s))-\bar{F}(\bar{U}^\gamma)\big]
\\&{} \otimes P_0^\gamma\big[
F(\bar{U}^\gamma+\tilde{\eta}(0))-\bar{F}(\bar{U}^\gamma)\big]\,ds \,,
\end{align*}
and $\bar{\beta}(t)=(\bar{\beta}_j(t))$ is an $M$ dimensional standard Brownian motion. For any~$T>0$ and any $\kappa>0$\,, there is a positive constant~$C_{\kappa,T}$ such that
\begin{equation*}
\mathbb{P}\left\{{\textstyle\sup_{0\leq t'\leq
T}} |\widetilde{U}^\gamma(t')-\bar{U}^\gamma(t')|\leq \gamma^{1+}
C_{k,T} \right\}\geq 1-\kappa\,.
\end{equation*}
Then for the original system~\eqref{e:uj}, by~\eqref{e:small fields}, we have the following reduced equation
\begin{eqnarray}
dU^\gamma_0(t)&=&\left[\mathcal{L}_\gamma
U^\gamma_0(t)+\alpha\gamma^2U_0^\gamma(t)+\alpha
P^\gamma_0\bar{F}(U^\gamma_0(t))\right]dt \nonumber\\&&{}+\sigma
\gamma dB_0^\gamma(t)
+\alpha\gamma^3\sqrt{\bar{Q}(\gamma^{-1}U_0^\gamma(t))}\,d\check{\beta}(t)
\,, \label{e:aver-equ}
\end{eqnarray}
where $\check{\beta}(t)=(\check{\beta}_j(t))=\gamma^{-1}\bar{\beta}(\gamma^2t)$ is the scaled $M$~dimensional standard Browian motion. Then for any $\kappa>0$\,, there is a constant $C_{\kappa, T}>0$ such that
\begin{equation*}
\mathbb{P}\left\{{\textstyle\sup_{0\leq t\leq \gamma^{-2}T}}|U^\gamma(t)-
U^\gamma_0(t)|\leq \gamma^{2+} C_{\kappa,T} \right\}\geq 1-\kappa\,.
\end{equation*}

Now having the reduced system~\eqref{e:aver-equ}, we use the approximation to the amplitude on each element to derive a further approximate model. Write $U^\gamma_0(x,t)=(u^\gamma_{j0}(x, t))$ and define
\begin{equation*}
(U_j(t))=a_0^\gamma(t)\left(e_{j,0}^\gamma(X_j)\right).
\end{equation*}
Then by~\eqref{e:ej0} for $(u^\gamma_{j,0}(x, t))$ we have the
following asymptotic expansion
\begin{eqnarray}\label{e:uj0}
u^\gamma_{j,0}(x,t)
\nonumber&=&a_0^\gamma(t)e_{j,0}^\gamma(X_j)+\gamma a_0^\gamma(t)
F^\gamma_{j,1}(x)+\gamma^2a_0^\gamma(t)F^\gamma_{j,2}(x)+\Ord{\gamma^3}
\nonumber\\
&=&
  \begin{cases}
   U_j(t)+\frac{\gamma}{h}\big(U_j(t)-U_{j-1}(t)\big)(x-X_j)
  \\\quad{} +\frac{\gamma^2}{2h^2} (U_{j-1}(t)-2U_j(t)+U_{j+1}(t))(x-X_j)(x-X_{j-1})
 \\\quad{}+\Ord{\gamma^3},\qquad X_{j-1}\leq x\leq X_j\,,\\
   U_j(t)+\frac{\gamma}{h} (U_{j+1}(t)-U_j(t))(x-X_j)\\\quad
   +\frac{\gamma^2}{2h^2}(U_{j-1}(t)-2U_j(t)+U_{j+1}(t))(x-X_j)(x-X_{j+1})
   \\\quad{}+\Ord{\gamma^3},\qquad X_j \leq x\leq X_{j+1}\,,
  \end{cases}
\end{eqnarray}
Putting~\eqref{e:uj0}  into~\eqref{e:aver-equ} yields
\begin{eqnarray}\label{e:U1-gamma}
dU_j(t)&=&\frac{\gamma^2}{h^2}\left(U_{j+1}(t)-2U_j(t)+U_{j-1}(t)
\right)dt+\gamma^2 \hat{\alpha}_jU_j(t)-\alpha U_j^3(t)
\nonumber\\
&&{}+\sigma \gamma\sqrt{q^h_{j,0}}e^\gamma_{j,0}(x)d\beta_{j,0}(t)
+\alpha\gamma^23\sqrt{2Q_j}U_j(t)e_{j,0}^\gamma(x)d\check{\beta}_j(t)
\nonumber\\&&{} +\Ord{\gamma^3, \alpha^2},
\end{eqnarray}
where
\begin{eqnarray*}&&
\hat{\alpha}_j=\alpha-3\alpha
\sigma^2\sum_{k=1}^\infty\frac{q^h_{j,k}}{2\lambda_k}
 (e_{j,0}(X_j))^2\,,
\\&&
Q_j=\int_0^\infty\mathbb{E}\left[\langle
\tilde{\eta}^2_{j,0}(s)-\mathbb{E}\tilde{\eta}^2_{j,0},e_{j,0}\rangle
\langle
\tilde{\eta}^2_{j,0}(0)-\mathbb{E}\tilde{\eta}^2_{j,0},e_{j,0}\rangle\right]\,ds\,.
\end{eqnarray*}
Here we use the approximation of $e_{j,0}(x)$ to $e^\gamma_{j,0}(x)$
for small $\gamma$\,.

Notice that system~\eqref{e:U1-gamma} is not a complete discrete
model because the noise terms are still described on the
mode~$(e_{j,0}^\gamma(x))$.  In order to give a discrete
approximating model for small coupling~$\gamma$, we explore the
evolution of the amplitude of the basic mode~$(e_{j,0}(x))$. For
this we project~$(U_j)$ onto the basic space~$E_0$ spanned
by~$(e_{j,0}(x))$. However, the fast modes~$V^\gamma$ have a nonzero
projection in basic space~$E_0$; because of the complicated
expression for~$(e_{j,k}^\gamma(x))$, we choose to
project~$\gamma\tilde{\eta}$, which approximates~$V^\gamma$ up to
error of~$\mathcal{O}(\gamma^2)$. For this we first
project~$\gamma\tilde{\eta}$ onto $(e_{j,0}^\gamma(x))$, then
project to~$E_0$.

Notice that for small coupling~$\gamma$, $\frac{1}{\gamma}\tilde{\eta}(t')$ behaves as a noise process. By a martingale approach~\cite{Kes79, WR08, Wata88} we have the following lemma.
\begin{lemma}
Assume bound~\eqref{e:Q-assumption}.  Then $\frac{1}{\gamma}\int_0^{t'}\tilde{\eta}(s)\,ds$ converges in distribution to
\begin{equation*}
\sigma\sum_{k=1}^\infty\sqrt{\frac{1}{\lambda_k}}\left(\sqrt{q^h_{j,k}}e_{j,k}(x)\,\tilde{\beta}_{j,k}(t')
\right) \quad\text{as }\gamma\rightarrow 0\,,
\end{equation*}
where $(\tilde{\beta}_{j,k}(t'))$\,, $k=1,2\ldots$\,, are mutually
independent standard $M$ dimensional Brownian motion in time
scale~$t'$.
\end{lemma}
Then on the right-hand sides of~\eqref{e:U1-gamma} there are additional noise forcing terms when projected to the basic space~$E_0$: namely
\begin{eqnarray*}
&&\gamma\sigma\sum_{k=1}^\infty\sqrt{\frac{q^h_{j,k}}{\lambda_k}}\left[\langle
e_{j,k}(x), e_{j,0}^\gamma(x)\rangle\,d\hat{\beta}_{j,k}(t)\right]e^\gamma_{j,0}(x)\\
\end{eqnarray*}
where
$(\hat{\beta}_{j,k}(t))=(\gamma^{-1}\tilde{\beta}_{j,k}(\gamma^2
t))$\,, $k=1,2\,,\ldots$\,, are mutually independent scalar standard
Brownian motions. By the expansion of
$(e_{j,0}^\gamma(x))$\,,~\eqref{e:ej0} and the expression of
$(e_{j,k}(x))$\,,
\begin{eqnarray*}
\langle e_{j,k}(x), e_{j,0}^\gamma(x)\rangle&=&\gamma\langle
e_{j,k}(x), F^\gamma_{j,1}(x)\rangle+\mathcal{O}(\gamma^2).
\end{eqnarray*}
Then define
\begin{eqnarray*}
\hat{\beta}^\gamma_{j,0}(t)&=&\sum_{k=1}^\infty
\sqrt{\frac{q^h_{j,k}}{\lambda_k}}\langle e_{j,k}(x),
F^\gamma_{j,1}(x)\rangle \hat{\beta}_{j,k}(t),
\end{eqnarray*}
and for $j=1\,, 2\,, \ldots\,, M$\,, $i=-1\,, 0\,, 1$
\begin{eqnarray*}
&& B_{j,i}(t)=\sqrt{q^h_{j,0}}\beta_{j,0}(t)e_{j-i,0}(X_{j-i}),\quad
\hat{B}^\gamma_{j,i}(t)=\hat{\beta}^\gamma_{j,0}(t)e_{j-i,0}(X_{j-i}),
\\&&
\check{B}_{j,i}(t)=\sqrt{Q_j}\check{\beta}_j(t)e_{j-i,0}(X_{j-i}).
\end{eqnarray*}
Projecting  the above system onto the basic space~$E_0$ and
by~\eqref{e:E-eta} we then have the following macroscopic discrete
approximation model to the \spde~\eqref{e:uj}--\eqref{cbd2} for
small coupling $\gamma>0$
\begin{eqnarray}
dU_j(t)&=&\frac{\gamma^2}{h^2}\left(U_{j-1}(t)-2U_j(t)+U_{j+1}(t)\right)dt+\gamma^2\hat{\alpha}_j
U_j(t)-\alpha U_j^3(t)\nonumber\\ &&{}+ \sigma\gamma
dB_{j,0}(t)+\sigma\gamma^2
d\hat{B}^\gamma_{j,0}(t)+3\gamma^2\sqrt{2}U_j(t)d\check{B}_{j,0}(t)\nonumber\\
&&{}+\frac{\sigma\gamma^2}{4}\left(dB_{j,1}(t)-2d
B_{j,0}(t)+dB_{j,-1}\right) \nonumber\\&&{}
+\frac{\sigma\gamma^3}{4}\left(d\hat{B}^\gamma_{j,1}(t)-2d\hat{B}^\gamma_{j,0}(t)+d\hat{B}^\gamma_{j,-1}\right)\nonumber\\&&{}+
\frac{3\sqrt{2}}{4}\gamma^3U_j(t)\Big(d\check{B}_{j,1}(t)-2d\check{B}_{j,0}(t)
+d\check{B}_{j,-1}\Big) \nonumber\\&&{}
+\Ord{\gamma^3,\alpha^2}.\label{e:Uj-gamma-model}
\end{eqnarray}
Furthermore by (\ref{e:Fj1}),
$F_{j,1}^\gamma=\mathcal{O}(\gamma)$\,. Then a truncation
of~\eqref{e:Uj-gamma-model} to errors $\mathcal{O}(\gamma^3,
\alpha^2)$ and evaluating at full coupling~$\gamma=1$ yields the
following macroscopic discrete system of \sde{}s, $j=1,\ldots,M$\,,
\begin{eqnarray}
dU_j(t)&\approx&\frac{1}{h^2}\left(U_{j-1}(t)-2U_j(t)+U_{j+1}(t)\right)dt+(\hat{\alpha}_j
U_j(t)-\alpha U_j^3(t))\,dt\nonumber\\ &&{}+ \sigma dB_{j,0}(t)+
3\sqrt{2}U_j(t)d\check{B}_{j,0}(t)\nonumber\\
&&{}+\frac{\sigma}{4}\Big(dB_{j,1}(t)-2d B_{j,0}(t)+dB_{j,-1}
\Big).\label{e:Uj-model}
\end{eqnarray}
This system of \sde{}s reduces to the system~\eqref{e:Uj} discussed
in the Introduction as a discrete model of the reaction-diffusion
\spde~\eqref{e:sRDe}.

\section{Consistency of macroscopic discrete model}\label{sec:consistency}
Next we study the consistency of the macroscopic discrete model~\eqref{e:Uj-model} by the definition of the Wiener processes~$W^\gamma(x, t)$ in Section~\ref{sec:coulping BDC}, as~$h$, the size of each element, converges to zero.

By the properties of $e_{j,k}(x)$ and~$e_{j,k}^\gamma(x)$ for small
coupling $\gamma>0$ in section~\ref{sec:coulping BDC} and
section~\ref{sec:asy} respectively, we have
\begin{align}
&\sqrt{2q^h_{j,0}}\beta_{j,0}(t)e_{j,0}^\gamma(x)=\left\langle
W(x,t), e_{j,0}^\gamma(x)\right\rangle
e_{j,0}^\gamma(x)/\|e_{j,0}^\gamma\|_0\nonumber
\\&{}= \sum_{k=0}^\infty\sqrt{q_k}\beta_k(t)\left< e_k(x),
e_{j,0}^\gamma(x)\right> e_{j,0}^\gamma(x)/\|e_{j,0}^\gamma\|_0\nonumber\\
&{}=\sum_{k=0}^\infty\sqrt{q_k}\beta_k(t)\left< e_k(x),
e^\gamma_{j,0}(X_j)+\gamma F_{j,1}(x)+\gamma^2 F_{j,2}(x)+F_{j,3}^\gamma(x) \right> e^\gamma_{j,0}(x)/\|e_{j,0}^\gamma\|_0\nonumber\\
&{}=\sum_{k=0}^\infty\sqrt{q_k}\beta_k(t)\left< e_k(x),
e_{j,0}(X_j)+\gamma F_{j,1}(x)+\gamma^2 F_{j,2}(x)+F_{j,3}^\gamma(x) \right> e_{j,0}(x)+\Ord\gamma  \nonumber\\
&{}=\sqrt{\frac{q_0}{L}}\beta_0(t)+\sum_{m=1}^\infty\sqrt{\frac{2q_{2m}}{L}}
\beta_{2m}(t)\frac{L}{2m\pi}\frac{1}{2h}\left[\sin\frac{2m\pi}{L}X_{j+1}-\sin\frac{2m\pi}{L}X_{j-1}\right]\nonumber\\
&\quad{}+\sum_{m=1}^\infty\sqrt{\frac{2q_{2m-1}}{L}}
\beta_{2m-1}(t)\frac{L}{2m\pi}\frac{1}{2h}\left[\cos\frac{2m\pi}{L}X_{j-1}-\cos\frac{2m\pi}{L}X_{j+1}\right]
\nonumber\\&\quad{}
+\mathcal{O}(h, \gamma)\nonumber\\
&{}=\sum_{k=0}^\infty\sqrt{q_k}\beta_k(t)e_{k}(X_j)+\Ord{h,\gamma}.
\label{e:beta1}
\end{align}
And similarly we have for $l=1,2,\ldots$
\begin{eqnarray}\label{e:beta2}
&&\sqrt{q^h_{j,l}}\beta_{j,l}(t)e_{j,l}^\gamma(x)=\left< W(x,t),
e_{j,0}^\gamma(x)\right>
e_{j,l}^\gamma(x)/\|e_{j,l}^\gamma\|_0\nonumber\\&=&
\sum_{k=1}^\infty\sqrt{q_k}\beta_k(t)\left< e_k(x),
e_{j,l}^\gamma(x)\right> e_{j,l}^\gamma(x)/\|e_{j,l}^\gamma\|_0\nonumber\\
&=&\sum_{k=1}^\infty\sqrt{q_k}\beta_k(t)\left< e_k(x),
e_{j,l}(x)\right> e_{j,l}(x)+\Ord\gamma\nonumber\\
&=&\Ord{h, \gamma}.
\end{eqnarray}

\begin{theorem}\label{thm:consistency}
The macroscopic discrete model~\eqref{e:Uj-model} is consistent to the original stochastic reaction-diffusion equation~\eqref{e:sRDe}--\eqref{e:sRDe-bd}.
\end{theorem}
\begin{proof}
Following the exact same discussion by  Roberts et al.~\cite{Rob09} for deterministic systems, we just consider the stochastic terms and $\hat{\alpha}_j$\,.

By the definition of~$\hat{\alpha}_j$, since $q_{jl}^h=\Ord{h}$ and
noticing that $\lambda_k=k^2\pi^2/h^2$\,, then we have
\begin{equation*}
\hat{\alpha}_j=\alpha+\mathcal{O}(h^2), \quad h\rightarrow 0\,.
\end{equation*}
For the stochastic terms, notice that  $\tilde{\eta}_k$~converges to
zero in mean square with speed
$\exp\{-\lambda_k\}=\exp\{-k^2\pi^2/h^2\}$ as $h\rightarrow 0$\,.
Then since $q_{j,l}^h=\Ord{h}$\,, we have $Q_j= \Ord{h^2}$   as
$h\rightarrow 0$\,. Then by the definition  of $W^\gamma(x,t)$ and
analysis~\eqref{e:beta1}--\eqref{e:beta2} the macroscopic discrete
model~\eqref{e:Uj-model} is consistent up error~$\Ord{h}$ to the
following stochastic reaction-diffusion equation
\begin{eqnarray*}
\p_tu(x,t)&=&\p_{xx}u(x,t)+\alpha (u(x,t)-u^3(x,t))+\sigma B(x,t),\quad x\in[0, L]\\
u(0,t)&=&u(L,t)\quad t\geq 0
\end{eqnarray*}
 where $B(x,t)$ distributes the same as $W(x,t)$\,. This completes the proof.
\end{proof}

\section{Conclusion}
Stochastic averaging is an effective method to extract
macroscopic dynamics from \textsc{spde}s with separated time
scale~\cite{WR08,WR08-1}. Here by applying  the stochastic averaging
and dividing the spatial domain into overlapping finite sized
elements with special interelement coupling boundary
conditions~(\ref{cbd1})--(\ref{cbd2}), we derive a macroscopic
discrete model (\ref{e:Uj-model}) for stochastic reaction-diffusion
partial differential equations~(\ref{e:sRDe}) with periodic boundary conditions. The most
important property  of such interelement coupling boundary
conditions is preserving the self-adjoint symmetry which is often so
important in application~\cite{Rob09}. Furthermore, by the choice  of
stochastic forcing on each element, this coupling boundary conditions
also assures the consistency for vanishing element size, section~\ref{sec:consistency}.

Moreover, the final discrete model (\ref{e:Uj-model}), which is
different from the usual finite difference approximation
model~(\ref{e:diff-appr}),  shows the importance of the subgrid scale interaction
between noise and spatial diffusion and provides a new rigorous
approach to constructing semi-discrete approximations to stochastic
reaction-diffusion \spde{}s.

\paragraph{Acknowledgements} This research is supported by the Australian Research Council grant DP0774311 and NSFC grant 10701072.



\end{document}